\documentclass[psamsfonts]{amsart}

\usepackage{amssymb,amsfonts}
\usepackage[all,arc]{xy}
\usepackage{enumerate}
\usepackage{mathrsfs}
\usepackage{mathtools}
\DeclarePairedDelimiterX\set[1]\lbrace\rbrace{#1}
\newtheorem{theorem}{Theorem}[section]
\newtheorem{corollary}[theorem]{Corollary}

\theoremstyle{definition}

\theoremstyle{remark}
\newtheorem{remark}[theorem]{Remark}

\let\phi=\varphi

\let\oldbigwedge\bigwedge
\def\BIGwedge{{\textstyle\oldbigwedge}}
\def\medwedge{{\scriptstyle\oldbigwedge}}
\def\bigwedge{\mathchoice{\BIGwedge}{\BIGwedge}{\medwedge}{}}

\makeatletter

\let\epsilon=\varepsilon

\let\c@equation\c@theorem
\makeatother
\numberwithin{equation}{section}

\begin{document}
\title{On square numbers of some special forms}
	
\author{\bfseries F. Jokar}
	
\address{Faculty of Mathematics\\ Ferdowsi University of Mashhad\\ Mashhad\\ Iran}
	
\email{jokar.farid@gmail.com}

	

\begin{abstract}

We show that there are infinitely many square numbers , which are constrocted by putting two square numbers together , that none of them are divisible by $10$ . We also find out some special square numbers.

\end{abstract}
	
\maketitle
	
\section{Introduction}
It is easy to see that there are infinitely many square numbers which are constructed by putting two square numbers together. For example , $4$ and $9$ are two square numbers. By putting $4$ and $9$ together, we construct $49$ which is another square number. For every $k \in \mathbb{N} $ , $49 \times 10^{2k}$ is a square number which is constructed by putting  $9 \times 10^{2k}$ and $4$ together , in which both $9 \times 10^{2k}$ and $4$ are square numbers. The question is that can we find infinitely many square numbers that none of them are divisible by $10$ and all of them are constructed by putting two square numbers together? In this paper, we prove that not only we can find such numbers ,but also we can find infinitely many square numbers which are constructed by putting two square numbers together , none of them are divisible by $10$ . We also prove that there are infinitely 2-tuples $(c,d)$ that $c$ and $d$ are square numbers , each of them are constructed by putting two square numbers together , that none of them are divisible by $10$ , and also $c|d$ . finally, we prove that for each $s \in \mathbb{N}$ , there are infinitely square numbers , which are constructed by putting $s$ zeros between two square numbers, none of them are divisible by $10$. 

\section{square numbers of some special forms}

\begin{theorem}
There are infinitely many square numbers , which are constructed by putting two square numbers together, none of them are divisible by $10$.
\end{theorem} 
\begin{proof}
Frist, consider the following Diophantine equation on natural numbers:
$$ x^2 + a^2 = z^2 $$
Note that if there is $\alpha \in \mathbb{N}$, such that
$10^{2\alpha} \mid \mid x^{2}$, also the number of digits of $a^2$ is $2\alpha$ and $10 \nmid a$, then $z^2$ is a square number which is constructed by putting two square numbers together , none of them are divisible by $10$.
Now, suppose that $\mathbb{Z}=x+n$. For every $\alpha \in \mathbb{N}$, for both of the following cases,  $z^2$ is a square number which is constructed by putting two square numbers together , none of them are divisible by $10$.

\begin{enumerate}
\item ($0\leq \beta \leq \alpha$), $k=5^{\alpha}\times 2^{\alpha - \beta}-1$, $n=2^{\beta}, a=kn=5^{\alpha}\times 2^{\alpha} - 2^{\beta}$, $x=\frac{(a-n)(a+n)}{2n}$,
\item ($0\leq \beta \leq \alpha$), $k=5^{\alpha-\beta}\times 2^{\alpha}-1$, $n=2^{\beta}, a=kn=5^{\alpha}\times 2^{\alpha} - 5^{\beta}$, $x=\frac{(a-n)(a+n)}{2n}$.
\end{enumerate}

Because $a^{2}<10^{2\alpha}$ and in both cases for each arbitrary $\alpha$, we have respectively
\begin{enumerate}
\item min\{ $a^{2} : a=5^{\alpha}\times 2^{\alpha} - 2^{\beta}$ \} $>$ $5^{2\alpha -1}\times 2^{2\alpha -1}$,
\item min\{ $a^{2} : a=5^{\alpha}\times 2^{\alpha} - 5^{\beta}$ \} $>$ $5^{2\alpha -1}\times 2^{2\alpha -1}$.
\end{enumerate}


Therefore, the number of digits of $a^{2}$in both cases is $2\alpha$.
\end{proof}
\begin{theorem}	
There are infinitely many 2-tuples $(c, d)$ that each c and d is constructed by putting two square numbers together that none of them are divisible by 10. Besides, $c \mid d$.
	\begin{proof}
	Let $\alpha \in \mathbb{N}$, $\alpha>2$. Suppose that $k=5^{\alpha}\times 2$, $n=2^{\alpha -1}, a=kn=5^{\alpha}\times 2^{\alpha} - 2^{\alpha -1}$, and $x=\frac{(a-n)(a+n)}{2n}$. Let $x_0= \frac{x}{2}$ and $a_0= \frac{a}{2}$. Now, $x^2+a^2$ and $x_0^2 + a_0^2$ are two square numbers which are constructed by putting two square numbers together , none of them are divisible by $10$ and $(x_0^2+a_0^2) \mid (x^2+a^2)$.

	\end{proof}
\end{theorem}

\begin{remark}
\label{remark}
if $n \in \mathbb{N}$. then
\begin{enumerate}
\item \label{first} The number of digits of $\frac{n}{4}$ either equals to the number of digits of n or equals to the number of digits of n minus one.
\item In two successive dividig of n by 4, the number of digits of achieved result, at least one unit (regarding the first part of this remark, at most two units) is less than the number of digits of n.
\end{enumerate}
\end{remark}

\begin{theorem}	
For each $r \in \mathbb{N}$, there is a chain of numbers $
l_{1} \mid l_{2} \mid l_{3} ... l_{r-1} \mid l_{r}
$ in which $l_{i}$ for $0 \leq i \leq r$ are square numbers that are not divisible by 10. Also each $l_{i}$ is constructed by putting two square numbers together.
\end{theorem}

	\begin{proof}
    Suppose that $r \geq 2$ is a natural number. Now let
	\begin{align*}
	\alpha = 2^{3(r-2)+2}, a=5^{\alpha}\times 2^{\alpha} - 2^{\alpha -1}, n=2^{\alpha -1} , x=\frac{(a-n)(a+n)}{2n}=(5^{\alpha}-1)(2^{\alpha}\times 5^{\alpha})\\
	\Rightarrow x=(5-1)\times(5+1)\times(5^{2}+1)\times \dots \times (5^{2^{3(r-2)+1}}+1)\times 2^{\alpha}\times 5^{\alpha}
	\end{align*}
	Now we define
	\begin{align*}
	A=(5-1)(5+1)(5^{2}-1)(5^{2}+1) \dots  (5^{2^{3(r-2)+1}}+1)
	\end{align*}
	Therfore, it is clear that
	\begin{align*}
	4^{3(r-2)+1}\mid A^{2}, 4^{3(r-2)+1}\mid a^{2}
	\end{align*}
	Now, for each $1 \leq k \leq 3(r-2)+1$, we define
	\begin{align*}
	x_{1}^{2} = \frac{A^{2}}{4}\times 2^{2\alpha} \times 5^{2\alpha}, a_{1}^{2}=\frac{a^{2}}{4},\\
	x_{2}^{2} = \frac{A^{2}}{4^{2}}\times 2^{2\alpha} \times 5^{2\alpha}, a_{2}^{2}=\frac{a^{2}}{4^{2}},\\\
	\vdots \\
	x_{k}^{2} = \frac{A^{2}}{4^{3(r-2)+1}}\times 2^{2\alpha} \times 5^{2\alpha}, a_{k}^{2}=\frac{a^{2}}{4^{3(r-2)+1}}
	\end{align*}
	
It is clear that $x_{1}^{2}+a_{1}^{2}$ is a square number which is not only indivisble by 10. but also is constructed by putting two square numbers together, none of them are divisible by 10. Additionally, regarding Remark \ref{remark}, the number of digits of $a_{2}^{2}$ is one unit less that $2\alpha$. Therefore $x_{2}^{2}+a_{2}^{2}=\frac{a^{2}}{4}$ is not necessarily constructed by putting two square numbers together, but we can say that it is constructed by putting one zero between two square numbers, none of them are divisible by 10. In the process of induction which we for obtaining $a_{i}s$, the worst case for $a_{3}$ (the case in which the maximum number of dividing occurs but the minimum number of the elements of our chain reveals) is that the number of digits of $a_{3}^{2}$ is also one unit less than $2\alpha$. (Because in this case $x_{3}^{2}+a_{3}^{2}$ is also a square number which is constructed by putting one zero between two square numbers, none of them are divisible by 10. Hence, this number cannot be a number of our chain, but if it does not happen for $a_{3}^{2}$, then regarding Remark \ref{remark}, the number of digits of $a_{3}^{2}$ is two units less than $2\alpha$, and in this case. we can assume that $x_{3}^{2}+a_{3}^{2}$ is $l_{r-2}$ in our chain.) but in this case, the number of digits of $a_{4}^{2}$ is necessarily two units less than $2\alpha$. Therefore, $x_{4}^{2}+a_{4}^{2}$ is constructed by putting $a_{4}^{2}$ and $x_{4}^{2}\times 100$ together. Hence, we can assume that $x_{4}^{2}+a_{4}^{2}$ is $l_{r-2}$ in our chain. Generally, we should know that in the process of successive dividing, we obtain two number of our chain in the first dividing. For obtaining $r-2$ other numbers of our chain, the worst case is as the following
\begin{enumerate}
\item[$\bullet$] The number of digits of $\frac{a^{2}}{4^{2}}$ is necessarily one unit less than $2\alpha$.
\item[$\bullet$]The number of digits of $\frac{a^{2}}{4^{3}}$ is one unit less than $2\alpha$.
\item[$\bullet$]The number of digits of $\frac{a^{2}}{4^{4}}$ is necessarily two units less than $2\alpha$.
\item[$\bullet$]The number of digits of $\frac{a^{2}}{4^{5}}$ is three units less than $2\alpha$.
\item[$\bullet$]The number of digits of $\frac{a^{2}}{4^{6}}$ is three units less than $2\alpha$.
\item[$\bullet$]The number of digits of $\frac{a^{2}}{4^{7}}$ is necessarily four units less than $2\alpha$.\\
\vdots
\end{enumerate}
Therefore, in general after getting $l_{r-2}$, we can obtain another number of our chain by at most three times dividing. Hence, for a chain with the length of $r$, after $4^{3(r-2)+1}$ times of dividing, we can construct the chain completely. 
\end{proof}
Finally, the following Corollary can be concluded.
\begin{corollary}
Suppose that $s\in \mathbb{N}$, then there are infinitely many square numbers which are constructed by putting $s$ zeros between two square numbers, none of them are divisible by 10.
\end{corollary}

\end{document}